\documentclass[11pt]{amsart}

\usepackage{amsmath, amssymb, amscd, graphicx}

\usepackage[all]{xy}

\newtheorem{thm}{Theorem}[section]

\newtheorem{prop}[thm]{Proposition}
\newtheorem{ques}[thm]{Question}

\def\del{{\partial}}

\begin{document}

\title[Alternating links and left-orderability]%
{Alternating links and left-orderability}

\begin{abstract}

Let $L \subset S^3$ denote an alternating link and $\Sigma(L)$ its branched double-cover.  We give a short proof of the fact that the fundamental group of $\Sigma(L)$ admits a left-ordering iff $L$ is an unlink.  This result is originally due to Boyer-Gordon-Watson.

\end{abstract}

\author[Greene]{Joshua Evan Greene}

\address{Department of Mathematics, Columbia University\\ New York, NY 10027}

\thanks{Partially supported by an NSF Post-doctoral Fellowship.}

\email{josh@math.columbia.edu}

\maketitle

\section{A group presentation.}

Consider a link $L \subset S^3$ presented by a connected planar diagram.  Color its regions black and white in checkerboard fashion, and assign each crossing a sign as displayed in Figure \ref{f: crossingincidence}.  From this coloring we obtain the {\em white graph} $W = (V,E)$.  This is the planar graph with one vertex for each white region, one signed edge for each crossing where two white regions touch, and one arbitrary distinguished vertex $r$ (the {\em root}).

\begin{figure}[t]
\includegraphics[width=2in]{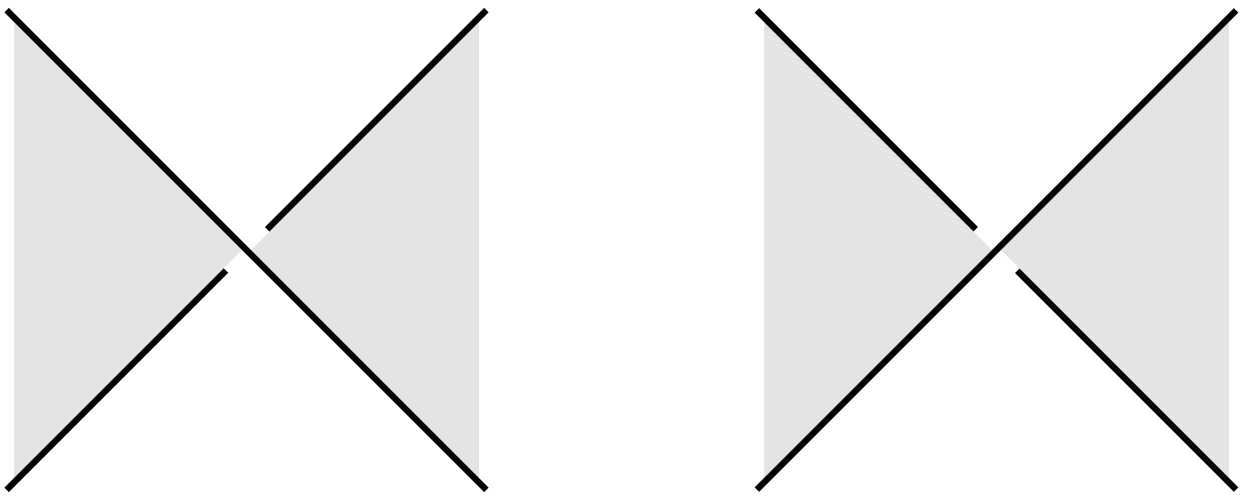}
\put(-125,-15){$+1$}
\put(-35,-15){$-1$}
\caption{Crossings and signs.}\label{f: crossingincidence}
\end{figure}

We form a group $\Gamma$ as follows.  It has one generator $x_v$ and one relation $r_v$ for each $v \in V$, as well as one additional relation $x_r$ (no confusion about the $r$'s!).  To describe the relation $r_v$, consider a small loop $\gamma_v$ centered at $x_v$ and oriented counter-clockwise.  Starting at an arbitrary point along $\gamma_v$, the loop meets edges $(v,w_1), \dots, (v, w_k)$ with respective signs $\epsilon_1, \dots, \epsilon_k$ in order; then $r_v = \prod_{i=1}^k (x_{w_i}^{-1}x_v)^{\epsilon_i}$.

\begin{prop}\label{p: presentation}

The fundamental group of $\Sigma(L)$ is isomorphic to $\Gamma$. \qed

\end{prop}

Proposition \ref{p: presentation} is established in \cite[$\S$3.1]{greene:spanningtrees}, in which the presentation for $\Gamma$ derives from a specific Heegaard diagram of $\Sigma(L)$ (see the beginning of \cite[$\S$3.2]{greene:spanningtrees}, as well).  Dylan Thurston points out that the standard derivation of the Wirtinger presentation of a knot group suggests an alternate route to this fact. 
%%%
%%%
%%%

\section{Non-left-orderability.}

In this section we use Proposition \ref{p: presentation} to establish the main result.

\begin{thm}[Boyer-Gordon-Watson \cite{bgw:lo}]\label{t: main}

If $L$ is an alternating link, then $\pi_1(\Sigma(L))$ admits a left-ordering iff $L$ is an unlink.

\end{thm}

\begin{proof}

Observe that if $L = L_1 \cup L_2$ is a split link, then $\Sigma(L) \cong \Sigma(L_1) \# \Sigma(L_2)$ and $\pi_1(\Sigma(L))$ decomposes as the free product $\pi_1(\Sigma(L_1)) * \pi_1(\Sigma(L_2))$.  Furthermore, a free product admits a left-ordering iff each of its factors do \cite{vinogradov:lo}.  Therefore, it suffices to restrict attention to the case of a non-split alternating link $L$.  With this assumption in place, the theorem follows once we establish that $\pi_1(\Sigma(L))$ admits a left-ordering iff $L$ is the unknot.

Present $L$ by a connected, alternating diagram; color it, distinguish a root $r$, and let $W$ denote the resulting white graph.  It follows that every edge gets the same sign $\epsilon$.  Mirroring $L$ if necessary (which leaves $\pi_1$ unchanged), we may assume that $\epsilon = 1$. Now suppose that $\Gamma \cong \pi_1(\Sigma(L))$ possessed a left-ordering $<$.  Choose a vertex $v$ for which $x_w \leq x_v$ for all $w \in V$.  If $x_v = x_w$ for {\em all} $w \in V$, then from the relation $x_r$ it follows that $1 = \Gamma \cong \pi_1(\Sigma(L))$; but then $1 = |H_1(\Sigma(L))| = \det(L)$, and since $L$ is alternating, it follows that $L = U$.

Thus, we assume henceforth that $L \ne U$ and seek a contradiction.  It follows that there exists some $w \in V$ for which $x_w < x_v$; from the connectivity of $W$, we may assume that $(v,w) \in E$.  It follows that $1 < x_w^{-1} x_v$, while $1 \leq x_{w_i}^{-1} x_v$ for every other edge $(v,w_i) \in E$.  Therefore, the product of all these terms in any order is greater than $1$.  In particular, $1 < \prod_{i=1}^k (x_{w_i}^{-1} x_v) = r_v = 1$, a contradiction.

\end{proof}

%%%
%%%
%%%

\section{Dicsussion.}

It remains an outstanding problem to relate $\pi_1(Y)$ to the Heegaard Floer homology of a 3-manifold $Y$.  As of this writing, it remains a possibility that a rational homology sphere $Y$ is an L-space iff $\pi_1(Y) \ne 1$ does not admit a left-ordering.  Theorem \ref{t: main} supports this conjecture, since $\Sigma(L)$ is a rational homology sphere L-space for a non-split alternating link $L$ (no confusion about the $L$'s!) \cite[Prop.3.3]{os:doublecover}.  Additional examples appear in \cite{bgw:lo, clw:graphlo, claywatson:lodehn, claywatson:locable, peters:lo}.

%The examples of Boyer-Gordon-Watson provide additional support.

In this spirit, Peter Ozsv\'ath raises an interesting question.   Let $(Y_0, Y_1, Y_2)$ denote a surgery triple of rational homology spheres.    That is, there exists a manifold $M$ with torus boundary and a triple of slopes $(\gamma_0,\gamma_1,\gamma_2)$ in $\del M$ such that $Y_i$ results from filling $M$ along slope $\gamma_i$ and $\gamma_i \cdot \gamma_{i+1} = +1$, for all $i \pmod 3$.  Cyclically permuting the indices if necessary, assume that $|H_1(Y_0)| = |H_1(Y_1)| + |H_1(Y_2)|$.

\begin{ques}\label{q}

If $\pi_1(Y_0)$ admits a left-ordering, does it follow that one of $\pi_1(Y_1)$ and $\pi_1(Y_2)$ must as well?

\end{ques}

\noindent Note that if $Y_1$ and $Y_2$ are L-spaces, then so is $Y_0$ according to the surgery exact triangle in $\widehat{HF}$.  This is the motivation behind Question \ref{q}.  An affirmative answer would imply that Theorem \ref{t: main} extends to quasi-alternating links.

\noindent {\em Note.} Tetsuya Ito has applied the idea in this paper to a different presentation for $\pi_1(\Sigma(L))$ to recover yet another proof of Theorem \ref{t: main} \cite{ito:bdc}.

\section*{Acknowledgments}

Thanks to Cameron Gordon for describing Theorem \ref{t: main} and the idea behind its original proof, which inspired the one presented here.  Thanks to Peter Ozsv\'ath, Dylan Thurston, and Liam Watson for further conversations.

%%%%%%%%%%%%%%%%
%%%%%%%%%%%%%%%%
%%%%%%%%%%%%%%%%

\bibliographystyle{plain}
\bibliography{/Users/Josh/Desktop/Papers/References}

\begin{thebibliography}{1}

\bibitem{bgw:lo}
S.~Boyer, C.~McA. Gordon, and L.~Watson.
\newblock On {L}-spaces and left-orderable fundamental groups.
\newblock {\em {\tt arXiv:1107.5016}}, 2011.

\bibitem{clw:graphlo}
A.~Clay, T.~Lidman, and L.~Watson.
\newblock Graph manifolds, left-orderability and amalgamation.
\newblock {\em {\tt arXiv:1106.0486}}, 2011.

\bibitem{claywatson:lodehn}
A.~Clay and L.~Watson.
\newblock Left-orderable fundamental groups and {D}ehn surgery.
\newblock {\em {\tt arXiv:1009.4176}}, 2010.

\bibitem{claywatson:locable}
A.~Clay and L.~Watson.
\newblock On cabled knots, {D}ehn surgery, and left-orderable fundamental
  groups.
\newblock {\em {\tt arXiv:1103.2358}}, 2011.

\bibitem{greene:spanningtrees}
J.~E. Greene.
\newblock A spanning tree model for the {H}eegaard {F}loer homology of a
  branched double-cover.
\newblock {\em To appear, Journal of Topology}.

\bibitem{ito:bdc}
T.~Ito.
\newblock Non-left-orderable double branched coverings.
\newblock {\em {\tt arXiv:1106.1499}}, 2011.

\bibitem{os:doublecover}
P.~Ozsv{\'a}th and Z.~Szab{\'o}.
\newblock On the {H}eegaard {F}loer homology of branched double-covers.
\newblock {\em Adv. Math.}, 194(1):1--33, 2005.

\bibitem{peters:lo}
T.~Peters.
\newblock On {L}-spaces and non left-orderable 3-manifold groups.
\newblock {\em {\tt arXiv:0903.4495}}, 2009.

\bibitem{vinogradov:lo}
A.~A. Vinogradov.
\newblock On the free product of ordered groups.
\newblock {\em Mat. Sbornik N.S.}, 25(67):163--168, 1949.

\end{thebibliography}

\end{document}